\documentclass[11pt]{amsart}

\usepackage[english]{babel}
\usepackage{amssymb,amsmath,amsthm}
\usepackage{dsfont}
\usepackage{graphicx}
\usepackage{color}

\calclayout

\newcommand{\N}{{\mathds N}}

\newcommand{\Z}{{\mathds Z}}

\theoremstyle{plain}\newtheorem{theo}{Theorem}[section]
\theoremstyle{plain}\newtheorem{assu}[theo]{Assumption}
\theoremstyle{plain}\newtheorem{lem}[theo]{Lemma}
\theoremstyle{plain}
\theoremstyle{plain}
\theoremstyle{definition}
\theoremstyle{definition}
\theoremstyle{definition}

\begin{document}

\title[Bootstrapping Covariance Operators]{Bootstrapping Covariance Operators of Functional Time Series}

\author[O.Sh. Sharipov]{Olimjon Sh. Sharipov}
\author[M. Wendler]{Martin Wendler}

\address{Otto-von-Guericke-Universit\"at Magdeburg, Germany}
\email{martin.wendler@ovgu.de}

\date{\today}

\begin{abstract} For testing hypothesis on the covariance operator of functional time series, we suggest to use the full functional information and to avoid dimension reduction techniques. The limit distribution follows from the central limit theorem of the weak convergence of the partial sum process in general Hilbert space applied to the product space. In order to obtain critical values for tests, we generalize  bootstrap results from the independent to the dependent case. This results can be applied to covariance operators, autocovariance operators and cross covariance operators. We discuss one sample and changepoint tests and give some simulation results.
\end{abstract}



\maketitle

\section{Introduction}\label{sec1}

\subsection{Literature overview}

In recent years, there has been a growing interest in statistical methods for functional data analysis, where the observations are often modeled as random variables taking values in a Hilbert space, see the book by Horv\'ath and Kokoszka \cite{horvath2012} for an introduction, and the articles by Goia and Vieu \cite{goia2016}, Aneiros et al. \cite{aneiros2019} for an overview of recent results. So called first order analysis of functional data deals with inference for the expectation, typically estimated by the sample mean.

Second order properties, meaning the covariance operators, of functional time series have been studied: Zhang and Sun \cite{zhang2010} proposed an $L_2$-norm based test for the hypothesis that the covariance functions of two functional populations are equal. Fremdt et al. \cite{fremdt2013} used dimension reduction via functional principal components for this testing problem. Jaru\v{s}kov\'a \cite{jaruskova2013} also studied two-sample tests for the covariance operator based on dimension reduction, but also a test for a change of the covariance operator at an unknown changepoint in a series of independent functional observations. The comparison of the covariance operators for mutliple samples was studied by Boente et al. \cite{boente2014} as well as Guo and Zhang \cite{guo2016}. Possible distances to compare covariance operators in function spaces were discussed by Pigoli et al. \cite{pigoli2014}.

Statistical inference for covariance operators has not only been studied for independent functional observations, but also for time series, starting with Bosq \cite{bosq2002}, who gave the first results on the empirical covariance operator for functional autoregressive processes. Panaretos and Tavakoli \cite{panaretos2013} have studied asymptotic properties of the estimated spectral density operator. Horv\'ath and Rice \cite{horvath2015} proposed a test for the hypothesis of independence against the alternative that the covariance operator is not equal to 0. Rice and Shum \cite{rice2017} also studied statistical procedures for the cross covariance operator, including a changepoint test. Two-sample tests for the covariance operators of two times series were considered by Zhang and Shao \cite{zhang2015}. Aue et al. \cite{aue2018} have studied tests for changepoint of the spectrum or the trace of covariance operators. Stoehr et al. \cite{stoehr2019} have developed a test for changes of the covariance operator and applied it to MRI data.

When studying functional data, it is common to use dimension reduction techniques like functional principal components. For example, in the context of changepoint detection for functional data, Berkes et al. \cite{berkes2009} used a projection on a finite number of functional principal components. Sharipov et al. \cite{sharipov2016} and Aue et al. \cite{aue2018a} proposed to use the full functional information without dimension reduction instead. While Aue et al. \cite{aue2018a} still rely on the estimation of eigenvalues to calculate critical values, Sharipov et al. \cite{sharipov2016} used bootstrap methods to estimate the distribution of the test statistic without such techniques. In the context of inference for the covariance operator, Boente et al. \cite{boente2014} have used bootstrap under the assumption of independence. Paparoditis and Sapatinas \cite{paparoditis2016} also assumed independence when studying bootstrap-based tests for covariance operators. Stoehr et al. \cite{stoehr2019} combined dimension reduction with block bootstrap to obtain critical values for changes in the covariance operator of functional time series. Recently,  Pilavakis, Paparoditis and Sapatinas \cite{pilavakis2019} have proved the validity of the moving block bootstrap for covariance operators of functional time series.

The aim of this paper is to show that results for bootstrapping the sample mean of functional times series (see \cite{dehling2015},\cite{sharipov2016}) can be extended to empirical covariance operators. This allows statistical inference for covariance operators without dimension reduction and without estimating eigenvalues. We will show that the (centered and rescaled) empirical covariance operator and its bootstrap counterpart converge to the same limit distribution. In a simulation study, we compare the nonoverlapping block bootstrap method to other methods proposed in the literature (\cite{rice2017}, \cite{pilavakis2019}). We will also study the sequential bootstrap, so that we can apply our theory to test for changes of the covariance operator in a functional time series. In the rest of this section, we discuss Hilbert-Schmidt operators, introduce our dependence conditions and describe the block bootstrap methods. The main results are given in Section \ref{sec2}. The finite sample performance is investigated in Section \ref{sec3} in a simulation study. Some auxiliary lemmas and the proofs of the main results follow in the last section.
 
\subsection{Hilbert-Schmidt operators}

Let $(X_n)_{n\in\Z}$ be a stationary time series with values in a separable Hilbert space $H$ with inner product $\langle \cdot,\cdot\rangle_H$ and norm $\|\cdot\|_H:=\sqrt{\langle \cdot,\cdot\rangle}_H$. In this paper, we are interested in studying different hypothesis on the covariance operator $V_X:H\rightarrow H$ of $X_n$, given by the relation
\begin{equation*}
\langle V_X(h_1),h_2\rangle_H=E\left[\langle X_n-EX_n,h_1\rangle_H\langle X_n-EX_n,h_2\rangle_H\right]
\end{equation*}
for $h_1,h_2\in H$. $E$ denotes the expectation of a random variable, whatever space it takes its values in. $V_X$ is an element of the product space $H\otimes H$, and this is a Hilbert space equipped with the Hilbert-Schmidt inner product $\langle V_1, V_2\rangle_{\operatorname{HS}}=\sum_{i=1}^\infty \langle V_1(b_i), V_2(b_i)\rangle_H$ for any orthonormal basis $(b_i)_{i\in\N}$ of $H$. The corresponding Hilbert-Schmidt norm will be denoted by $\|\cdot\|_{\operatorname{HS}}$. By Lemma \ref{lemMoment}, it follows that the covariance operator has a finite Hilbert-Schmidt norm almost surely, if $E\|X_i\|_H^2<\infty$. This will allow us to apply existing results on limit theorems and bootstrap consistency for the inference on the covariance operator.

The autocovariance operator $V_k$ to the lag $k$ is given by
\begin{equation*}
\langle V_k(h_1),h_2\rangle_H=E\left[\langle X_n-EX_n,h_1\rangle_H\langle X_{n+k}-EX_{n+k},h_2\rangle_H\right]
\end{equation*}
for $h_1,h_2\in H$. This falls into the framework of covariance operators if we consider the time series $(X_n,X_{n+k})_{n\in\Z}$ with values in the direct sum $H\oplus H$ equipped with the inner product given by $\langle (h_1,h_2),(h_3,h_4)\rangle_{H\oplus H} = \langle h_1,h_3 \rangle_H+\langle h_2,h_4 \rangle_H$. Let $V_{(X_0,X_k)}$ be the covariance operator of this time series, then $\langle V_k(h_1),h_2\rangle_H=\langle V_{(X_0,X_k)}(h_1,0),(0,h_2) \rangle_{H\oplus H}$.

Similarly, for a second time series $(Y_n)_{n\in\Z}$ with values in separable Hilbert space $G$ with inner product $\langle \cdot,\cdot\rangle_G$ and norm $\|\cdot\|_G:=\sqrt{\langle \cdot,\cdot\rangle}_G$, the cross covariance operator $V_{XY}: H\rightarrow G$ is given by
\begin{equation*}
\langle V_{XY}(h),g\rangle_G=E\left[\langle X_n-EX_n,h\rangle_H\langle Y_n-EY_n,g\rangle_G\right]
\end{equation*}
for $h\in H$, $g\in G$. In a similar way as above, we can consider the direct sum $H\oplus G$ with inner product $\langle\cdot,\cdot\rangle_{H\oplus G}$ and the covariance operator $V_{(X_n,Y_n)}$ of the time series $(X_n,Y_n)_{n\in\Z}$ . We observe that 
$\langle V_{XY}(h),g\rangle_G=\langle V_{(X_n,Y_n)}(h,0),(0,g) \rangle_{H\oplus G}$. So we do not have to deal with the three cases (covariance operator, autocovariance operator for lag $k$ and cross covariance operator) separately in our theoretical results.

We have that $V_X=E[(X_i-EX_i)\otimes (X_i-EX_i)]$, so a natural estimator of the covariance operator is the empirical covariance operator
\begin{equation*}
\hat{V}_n:=\frac{1}{n}\sum_{i=1}^n (X_i-\bar{X})\otimes (X_i-\bar{X}),
\end{equation*}
where $\bar{X}=\frac{1}{n}\sum_{i=1}^n X_i$ and $\otimes$ denotes the tensor product. Our aim is to prove the weak convergence of $\hat{V}_n$ (after centering and rescaling) in the space $(H\otimes H, \|\cdot\|_{\operatorname{HS}})$. Furthermore, we want to study the sequential version of the empirical covariance operator $\hat{V}_{[nt]}$, $t\in[0,1]$ (where $[.]$ denotes the integer part of a real number). This will allow us to apply our theory to changepoint problems.

\subsection{Dependence conditions}

Let $(X_n)_{n\in\Z}$ be a stationary sequence of $H$-valued random variables. We say that the sequence is $L_p$-near epoch dependent (NED) on a stationary sequence $(\xi_n)_{n\in\Z}$ (taking values in a separable space $S$), if for all $m\in\N$ there exists a function $f_m:S^{2m+1}\rightarrow H$, such that 
\begin{equation*}
E\left\|X_0-f_m(\xi_{-m},\ldots,\xi_m)\right\|^p_H\leq a_m 
\end{equation*}
and $a_m\rightarrow 0$ as $m\rightarrow \infty$. We call the sequence $a_m$, $m\in\N$ the approximation constants. In what follows, we will assume that the sequence $(\xi_n)_{n\in\Z}$ is absolutely regular ($\beta$-mixing). We define the coefficients of absolute regularity $(\beta_m)_{m\in\N}$ by
\begin{equation*}
\beta_m=\Big|E\sup_{A\in\mathcal{F}_{m}^\infty}\left(P(A|\mathcal{F}_{-\infty}^{0})-P(A)\right)\Big|,
\end{equation*}
where $\mathcal{F}_a^b:=\sigma(\xi_a,\xi_{a+1},\ldots,\xi_b)$ is the sigma-field generated by $\xi_a,\xi_{a+1},\ldots,\xi_b$.

The combination of two notions of weak dependence (absolute regularity and near epoch dependence) covers many relevant time series models. For instance, stationary autoregressive processes and GARCH(1,1) processes have an exponential decay of the approximation constants. Many dynamical systems are also covered, see Borovkova et al. \cite{borovkova2001} for details.

\subsection{Bootstrap}

As we will see, the limit of the empirical covariance operator will depend on the long run covariance operator of the sequence of tensor products $((X_n-EX_n)\otimes (X_n-EX_n))_{n\in\Z}$. It is not easy to estimate this infinite dimensional parameter. For this reason, we propose to use the nonoverlapping block bootstrap method introduced by Carlstein \cite{carlstein1986}, where we sample blocks of Hilbert space valued vectors without dimension reduction. The sample of length $n$ is divided in $k=[n/p]$ blocks $I_1,\ldots,I_k$ of length $p$:
\begin{equation*}
I_j=\left(X_{(j-1)p+1},X_{(j-1)p+2},\ldots,X_{jp}\right)
\end{equation*}
We choose $p=p_n$, such that $p\rightarrow\infty$ as $n\rightarrow\infty$ and $p_n/n\rightarrow 0$. Then we produce a new bootstrap sample $X^\star_{1},\ldots,X^\star_{kp}$ by drawing $k$ times with replacement from these blocks:
\begin{equation*}
P\left((X_{(i-1)p+1}^\star,X_{(i-1)p+2}^\star,\ldots,X_{ip}^\star)=I_j\right)=\frac{1}{k}\ \ \ \text{for} \ \ \ i,j=1,\ldots,k.
\end{equation*}
With $P^\star$ and $E^\star$, we denote the probability and expectation conditional on $X_1,\ldots,X_n$, so for $\bar{X}^\star_{n}:=\frac{1}{kp}\sum_{i=1}^{kp}X_i^\star$, we have
\begin{equation*}
E^\star\left[\bar{X}^\star_{n}\right]=E^\star\big[\frac{1}{p}\sum_{i=1}^{p}X_i^\star\big]=\frac{1}{k}\sum_{j=1}^k\frac{1}{p}\sum_{i=jp-p+1}^{jp}X_i=\frac{1}{kp}\sum_{i=1}^{kp}X_i.
\end{equation*}
Note that we can exchange the bootstrap procedure and the tensor product, so that $X_i^\star \otimes X_i^\star=(X_i \otimes X_i)^\star$.

\section{Main results}\label{sec2}

\subsection{Limit theorems}

For the times series $(Y_n)_{n\in\Z}$ with values in $H\otimes H$ given by
\begin{equation*}
Y_n:=(X_n-EX_n)\otimes(X_n-EX_n),
\end{equation*}
we define the long run covariance operator $C_\infty$ by
\begin{equation}\label{LRC}
\langle C_\infty(x),y\rangle_{\operatorname{HS}}=\sum_{k=-\infty}^{\infty}E\big[\langle Y_0,x\rangle_{\operatorname{HS}}\langle Y_k,y\rangle_{\operatorname{HS}}\big]
\end{equation}
for $x,y\in H\otimes H$. We will first state a central limit theorem for the empirical covariance operator
\begin{equation*}
\hat{V}_n:=\frac{1}{n}\sum_{i=1}^n (X_i-\bar{X})\otimes (X_i-\bar{X}).
\end{equation*}

\begin{theo}\label{limitfixed} Let $(X_n)_{n\in\Z}$ be a stationary sequence with marginal covariance operator $V_X$ and let it be $L_2$-NED with approximating constants $(a_m)_{m\in\N}$ on an absolutely regular process with mixing coefficients $(\beta_m)_{m\in\N}$, such that for some $\delta>0$
\begin{enumerate}
\item $E\big\|X_i\big\|_{H}^{4+\delta}<\infty$,
\item $\sum_{m=1}^\infty a_{m}^{\delta/(2+2\delta)}<\infty$,
\item $\sum_{m=1}^\infty \beta_m^{\delta/(4+\delta)}<\infty$.
\end{enumerate}
Then we have the weak convergence
\begin{equation}
\sqrt{n}\left(\hat{V}_n-V_X\right)\Rightarrow N(0,C_\infty),
\end{equation}
where $N(0,C_\infty)$ denotes the Gaussian distribution in $H\otimes H$ with mean 0 and covariance operator $C_\infty$ given by (\ref{LRC}).
\end{theo}

The idea to prove this and the other limit theorems will be to apply known results on weak convergence (or bootstrap) in general Hilbert spaces to the time series $(Y_n)_{n\in\Z}$ in the first step, and then to show that the empirical covariance operator can be approximated by partial sums of the process in a general Hilbert space to the time series $(Y_n)_{n\in\Z}$.

Next, we will give a sequential limit theorem which will be needed to detect changes in the covariance operator. For this aim, we define the process $(W_n(t))_{t\in[0,1]}$ with
\begin{equation*}
W_n(t):=\frac{[nt]}{\sqrt{n}}\left(\hat{V}_{[nt]}-V_X\right)
\end{equation*}
in the space $D_{H\otimes H}[0,1]$ of cadlag functions with values in $H\otimes H$.  

\begin{theo}\label{limitseq} Let $(X_n)_{n\in\Z}$ be a stationary sequence and let it be $L_2$-NED with approximating constants $(a_m)_{m\in\N}$ on an absolutely regular process with mixing coefficients $(\beta_m)_{m\in\N}$, such that for some $\delta>0$
\begin{enumerate}
\item $E\big\|X_i\big\|_{H}^{8+\delta}<\infty$,
\item $\sum_{m=1}^\infty m^2 a_{m}^{\delta/(12+2\delta)}<\infty$,
\item $\sum_{m=1}^\infty m^2\beta_m^{\delta/(8+\delta)}<\infty$.
\end{enumerate}
Then we have the weak convergence
\begin{equation}
W_n\Rightarrow W
\end{equation}
where $W$ is a Brownian motion with values in $H\otimes H$, such that the increments $W(t)-W(s)$ have the covariance operator $|t-s|C_\infty$.
\end{theo}

The limit theorem for the sequential empirical covariance operator can be used to deduce the asymptotic distribution of the CUSUM statistic $CS_n:=\sup_{t\in[0,1]}\|W_n(t)-tW_n(1)\|_{\operatorname{HS}}$, which is used to obtain changepoints, see section \ref{sec3}.

\subsection{Bootstrap consistency}

The bootstrap version of the empirical covariance operator is defined as 
\begin{equation*}
V_n^{\star}:=\frac{1}{kp}\sum_{i=1}^{kp} (X_i^\star-\bar{X}^\star)\otimes (X_i^\star-\bar{X}^\star),
\end{equation*}
with $\bar{X}^\star=\frac{1}{kp}\sum_{i=1}^{kp}X_i^\star$. To guarantee the conditional weak convergence of the bootstrap version, we need the following assumption on the block length:
\begin{assu}\label{assublock} Assume that the block length $p=p_n$ is nondecreasing and satisfies
\begin{enumerate}
\item $p_n\rightarrow \infty$ and $p_n/n\rightarrow 0$ as $n\rightarrow\infty$,
\item $p_n=O(n^{1-\epsilon})$ for some $\epsilon>0$,
\item $p_{2^{l-1}+1}=p_{2^{l-1}+2}=\ldots=p_{2^l}$ for all $l\in\N$.
\end{enumerate}
\end{assu}

Under this conditions on the block length and dependence conditions similar to the conditions of Theorem \ref{limitfixed}, we obtain bootstrap consistency:

\begin{theo}\label{bootfixed} Let $(X_n)_{n\in\Z}$ be a stationary sequence and let it be $L_2$-NED with approximating constants $(a_m)_{m\in\N}$ on an absolutely regular process with mixing coefficients $(\beta_m)_{m\in\N}$, such that for some $\delta>0$ and $\delta'\in(0,\delta)$
\begin{enumerate}
\item $E\big\|X_i\big\|_{H}^{4+\delta}<\infty$,
\item $\sum_{m=1}^\infty a_{m}^{\delta'/(2+2\delta')}<\infty$ and $\sum_{i=1}^\infty m^{3/2}a_m^{1/2}<\infty$,
\item $\sum_{m=1}^\infty \beta_m^{\delta'/(4+\delta')}<\infty$ and $\sum_{m=1}^\infty m\beta_m<\infty$.
\end{enumerate}
If Assumption \ref{assublock} holds for the block length, we have 
\begin{equation}
\sqrt{pk}\left(V_n^\star-\hat{V}_n\right)\Rightarrow^\star N(0,C_\infty),
\end{equation}
where $N(0,C_\infty)$ is the Gaussian distribution with covariance operator $C_\infty$ given by (\ref{LRC}), and where $\Rightarrow^\star$ denotes weak convergence conditional on $(X_n)_{n\in Z}$ in probability.
\end{theo}

This theorem justifies to use the empirical quantiles obtained by the Monte Carlo method as critical values for testing hypothesis regarding the empirical covariance operator. Of course, the bootstrap could also be used to obtain confidence balls for the covariance operator. Let us note that Theorem 2.1 of Pilavakis, Paparoditis and Sapatinas \cite{pilavakis2019} gives a similar result for the moving block bootstrap, with stronger dependence condition (4-approximability instead of $L_2$-NED), but weaker assumptions on the rate of approximation constants. The proof of our Theorem shows that the consistency for the empirical covariance operator follows from the consistency of the bootstrap for partial sums of $(Y_n)_{n\in\N}$ and $(X_n)_{n\in\N}$. So it should be possible to extend other bootstrap methods in Hilbert spaces in the same way to empirical covariance operators: the stationary bootstrap studied by Politis and Romano \cite{politis1994}, the dependent wild bootstrap (see Bucchia, Wendler \cite{bucchia2017}), the functional AR-sieve bootstrap proposed by Paparoditis \cite{paparoditis2018}, or the bootstrap based on functional principal components introduced by Shang \cite{shang2018}.

Of course the choice of the block length $p$ plays an important role in practice. For functional time series, the optimal bandwidth for the kernel estimation of the long-run covariance operator has been studied by Berkes, Horv\'ath and Rice \cite{berkes2016}, and a data adaptive method to choose the bandwidth has been proposed by Rice and Shang \cite{rice2017}. As the implicit long-run covariance estimator in the bootstrap procedure is similar to the kernel estimator with Bartlett kernel, we will use their method in our simulation study.

To obtain critical values for the CUSUM statistic, we need a bootstrap analogue of Theorem \ref{limitseq}. We define the bootstrap version $W_n^\star$ of the sequential empirical covariance operator by
\begin{equation*}
W_n^\star(t):=\frac{1}{\sqrt{pk}}\sum_{i=1}^{[pkt]}\Big((X_i^\star-\bar{X}^\star)\otimes (X_i^\star-\bar{X}^\star)-\hat{V}_n\Big)
\end{equation*}
for $t\in[0,1]$. We get the following result on process convergence:

\begin{theo}\label{bootseq} Let $(X_n)_{n\in\Z}$ be a stationary sequence and let it be $L_2$-NED with approximating constants $(a_m)_{m\in\N}$ on an absolutely regular process with mixing coefficients $(\beta_m)_{m\in\N}$, such that for some $\delta>0$
\begin{enumerate}
\item $E\big\|X_i\big\|_{H}^{8+\delta}<\infty$,
\item $\sum_{m=1}^\infty m^2 a_{m}^{\delta/(12+2\delta)}<\infty$,
\item $\sum_{m=1}^\infty m^2\beta_m^{\delta/(8+\delta)}<\infty$.
\end{enumerate}
If Assumption \ref{assublock} holds for the block length, we have the weak convergence
\begin{equation}
W_n^\star\Rightarrow^\star W,
\end{equation}
conditional on $(X_n)_{n\in Z}$ in probability, where $W$ is a Brownian motion with values in  $H\otimes H$, such that the increments $W(t)-W(s)$ have the covariance operator $|t-s|C_\infty$.
\end{theo}

\section{Simulation results}\label{sec3}

\subsection{Cross covariance operator} We test the hypothesis that two functional time series $(X_n)_{n\in\Z}$ and $(Y_n)_{n\in\Z}$ are uncorrelated, that means that the cross covariance  operator satisfies $V_{XY}=0$. We are following the model of Rice and Shum \cite{rice2017}: Both time series take values in the Hilbert space $L^2[0,1]$ of square integrable functions with inner product $\langle x,y \rangle_H=\int x(s)y(s)ds$. For practical reasons, we use a finite equidistant grid of 100 points on $[0,1]$ to calculate integrals. The strength of correlation is modeled by a parameter $\alpha\in[0,1]$ and we have
\begin{align*}
X_i&:=\alpha\varepsilon_{c,i}+(1-\alpha)\varepsilon_{X,i},\\
Y_i&:=\alpha\varepsilon_{c,i}+(1-\alpha)\varepsilon_{Y,i},\\
\end{align*}
where $(\varepsilon_{c,i})_{i\in\Z}$, $(\varepsilon_{X,i})_{i\in\Z}$, $(\varepsilon_{Y,i})_{i\in\Z}$ are three independent time series. For $\alpha=0$, the two time series $(X_n)_{n\in\Z}$, $(Y_n)_{n\in\Z}$ are independent, while for $\alpha\neq 0$, the cross covariance operator satisfies $V_{XY}=\alpha^2 V_{\varepsilon_c}$. For the time series we consider two models: Let $W_{X,i}$, $W_{Y,i}$, $W_{c,i}$, $i\in\Z$ be independent standard Brownian motion and either
\begin{itemize}
\item $\varepsilon_{X,i}=W_{X,i}$, $\varepsilon_{Y,i}=W_{Y,i}$, $\varepsilon_{c,i}=W_{c,i}$ (IID case)
\item $\varepsilon_{X,i}(t)=\int\Phi(s,t)\varepsilon_{X,i-1}(s)ds+W_{X,i}(t)$, $\varepsilon_{Y,i}(t)=\int\Phi(s,t)\varepsilon_{Y,i-1}(s)ds+W_{Y,i}(t)$, $\varepsilon_{c,i}(t)=\int\Phi(s,t)\varepsilon_{c,i-1}(s)ds+W_{c,i}(t)$, where $\Phi(s,t)=\min\{s,t\}$ (FAR(1) case)
\end{itemize}
As a test statistic, we use the squared Hilbert-Schmidt norm of the empirical cross covariance operator
\begin{equation*}
S_n:=\int_0^1\! \int_0^1 \bigg(\frac{1}{n}\sum_{i=1}^n\left(X_i(s)-\bar{X}(s)\right)\left(Y_i(t)-\bar{Y}(t)\right)\bigg)^2dsdt
\end{equation*}
and use the block bootstrap version to calculate critical values. This test statistic is obviously a continuous functional of the empirical covariance operator of the joint time series $(X_n,Y_n)_{n\in\N}$, so we can apply Theorem \ref{bootfixed} and conclude that the bootstrap is consistent.

In the simulation study, we study the finite sample behavior for sample sizes $n=50$ or $n=100$. The results can be found in Table \ref{tab1} for $n=50$ and in Table \ref{tab2} for $n=100$. They are based on 3000 simulation runs. For each of the 3000 samples, we generate 1000 bootstrap samples to calculate the critical values. The empirical rejection frequencies are shown for theoretical sizes 1\%, 5\% and 10\% of the test. We have used fixed block lengths as well as data adaptive block lengths using the bandwidth selection method by Rice and Shang \cite{rice2017b}. We used an initial bandwidth $h=n^{1/5}$ and the Bartlett weight function, and we have rounded the resulting bandwidth up to the next integer to obtain a valid block length. Shorter block lengths lead to better sizes, but also the data adaptive choice of the block length works quite well. For $n=50$, the tests are oversized, but this problem is less pronounced for $n=100$.

In Figure \ref{fig1} we make a comparison between the methods propsed by  Rice and Shum \cite{rice2017} with and without dimension reduction (values obtained from Table 1 and Figure 2 of \cite{rice2017}), and with the moving block bootstrap used by Pilavakis, Paparoditis and Sapatinas \cite{pilavakis2019}. There is not much difference between the two bootstrap methods, but both have a somewhat better performance than the tests based on asymptotical critical values.

\begin{figure}
\includegraphics[width=10cm]{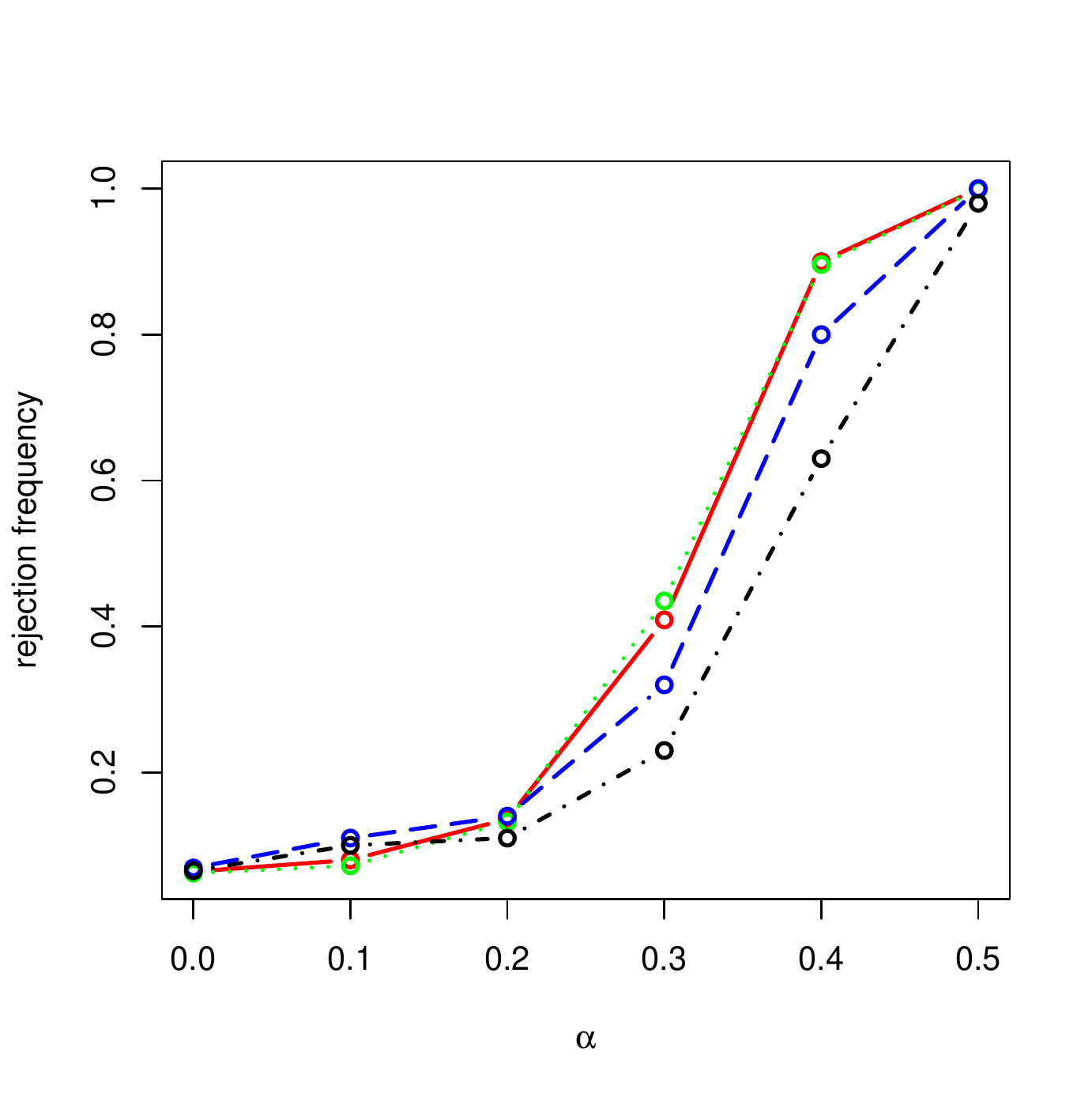}
\caption{Empirical size and power depending on $\alpha$ for the test based on $S_n$ and nonoverlapping block bootstrao (red solid), moving block bootstrap (green dotted), and for the two tests proposed by Rice and Shum \cite{rice2017}: test based on $F_n$ (blue dashed) and $F_{n,3}$ (black point dashed) for a FAR(1) time series of length $n=100$, theoretal size 5\%.}\label{fig1}
\end{figure}

\begin{table}
\caption{Empirical rejection frequencies for sample size $n=50$ under the hypothesis of no cross correlation ($\alpha=0$) and under the alternative ($\alpha>0$) for different block lengths $p$ and theoretical sizes 1\%, 5\%, 10\%.}
\begin{tabular}{cc|ccc|ccc|}\label{tab1}
 &  &  & IID &  &  & FAR(1) &  \\
 & \phantom{MMMM} & 1\% & 5\% & 10\% &1\% & 5\% & 10\% \\ 
\hline   &  & \phantom{MMMM} & \phantom{MMMM} & \phantom{MMMM} \vspace{-0.2cm}& \phantom{MMMM} & \phantom{MMMM} & \phantom{MMMM} \\ 
 $\alpha=0$ & $p=2$ & 0.015 & 0.061 & 0.123 & 0.014 & 0.075 & 0.132 \\ 
 & $p=3$ &  0.018 & 0.065 & 0.117 & 0.018 & 0.064 &  0.120 \\ 
 & $p=5$ & 0.026 & 0.087 & 0.148 & 0.033 & 0.098 & 0.161 \\ 
 & adaptive & 0.017 & 0.066 &  0.120 & 0.017 & 0.073 & 0.130 \\ 
\vspace{-0.2cm} &  &  &  &  &  &  &  \\ 
$\alpha=0.1$ &$ p=2$ & 0.018 & 0.074 & 0.137 & 0.017 & 0.067 & 0.126 \\ 
 & $p=3$ & 0.014  & 0.066 &  0.111 &  0.016 & 0.066 & 0.123 \\ 
 & $p=5$ & 0.037 & 0.104 & 0.166 &  0.035 & 0.109 & 0.173 \\ 
 & adaptive & 0.015 & 0.067 & 0.125  &   0.020 & 0.072 & 0.139  \\ 
\vspace{-0.2cm} &  &  &  &  &  &  &  \\ 
$\alpha=0.2$ & $p=2 $ & 0.019 & 0.082 & 0.149 & 0.028 & 0.101 & 0.163 \\ 
 & $p=3$ & 0.026 & 0.083 &  0.140 & 0.025 & 0.089 & 0.156
 \\ 
 & $p=5$ & 0.035 & 0.106 & 0.174 & 0.045 & 0.122 & 0.193 \\ 
 & adaptive & 0.025 & 0.086 & 0.144 &  0.030 & 0.105 &  0.180 \\ 
\vspace{-0.2cm} &  &  &  &  &  &  &  \\ 
$\alpha=0.3$ & $p=2$ & 0.072 & 0.212 & 0.318 &  0.087 & 0.226  & 0.340 \\ 
 & $p=3$ & 0.076 & 0.205 & 0.299 & 0.078 & 0.206 & 0.311\\ 
 & $p=5$ & 0.107 & 0.254 & 0.362 & 0.122 & 0.263 & 0.359 \\ 
 & adaptive &0.068 &  0.210 & 0.323 &   0.091 & 0.224 & 0.333  \\ 
\vspace{-0.2cm} &  &  &  &  &  &  &  \\ 
$\alpha=0.4$ & $p=2$ & 0.332 & 0.606 &  0.730 & 0.343 & 0.633 & 0.750 \\ 
 & $p=3$ & 0.309 & 0.578 & 0.698 & 0.342 &  0.614 &  0.730 \\ 
 & $p=5$  & 0.390 &  0.634 & 0.744 & 0.424 & 0.654 & 0.760 \\ 
 & adaptive &0.336  &  0.600 & 0.722 &   0.358 & 0.617 & 0.739  \\ 
\vspace{-0.2cm} &  &  &  &  &  &  &  \\ 
$\alpha=0.5$ & $p=2$ & 0.840 & 0.965 & 0.986 & 0.852 & 0.972  & 0.992 \\ 
 & $p=3$ & 0.788 & 0.953 & 0.981 & 0.819 &  0.960 & 0.983 \\ 
 & $p=5$ &  0.851 & 0.962 & 0.985 & 0.851 & 0.969 & 0.988 \\ 
 & adaptive &  0.825 & 0.958 & 0.983 &   0.840 & 0.962 & 0.985 \\ 
\end{tabular} 
\end{table}

\begin{table}
\caption{Empirical rejection frequencies for sample size $n=100$ under the hypothesis of no cross correlation ($\alpha=0$) and under the alternative ($\alpha>0$) for different block lengths $p$ and theoretical sizes 1\%, 5\%, 10\%.}
\begin{tabular}{cc|ccc|ccc|}\label{tab2}
 &  &  & IID &  &  & FAR(1) &  \\
 & \phantom{MMMM} & 1\% & 5\% & 10\% &1\% & 5\% & 10\% \\ 
\hline   &  & \phantom{MMMM} & \phantom{MMMM} & \phantom{MMMM} \vspace{-0.2cm}& \phantom{MMMM} & \phantom{MMMM} & \phantom{MMMM} \\ 
 $\alpha=0$ & $p=3$ & 0.013 & 0.061 & 0.116  &  0.018 & 0.066 & 0.121 \\ 
 & $p=5$ &  0.023 & 0.069 & 0.124  & 0.026 &  0.090 & 0.148 \\ 
 & $p=7$ & 0.021 & 0.072 & 0.128  &  0.022 & 0.076 & 0.136 \\ 
 & adaptive & 0.014 & 0.055 & 0.111 &  0.018 & 0.064 & 0.123 \\ 
\vspace{-0.2cm} &  &  &  &  &  &  &  \\ 
$\alpha=0.1$ & $p=3$ & 0.015 & 0.057 &  0.110  &  0.024 & 0.084   & 0.140 \\ 
 & $p=5$ & 0.027 & 0.077 & 0.137  &  0.024 & 0.086 & 0.147 \\ 
 & $p=7$ & 0.020 &  0.073 &  0.129  &  0.026 & 0.083 & 0.143 \\ 
 & adaptive & 0.010 & 0.056 & 0.114 &  0.020 &  0.080  & 0.145 \\ 
\vspace{-0.2cm} &  &  &  &  &  &  &  \\ 
$\alpha=0.2$ & $p=3$ &  0.029 & 0.101 & 0.175 & 0.047 & 0.128  & 0.210 \\ 
 & $p=5$ & 0.039 & 0.108  & 0.180  &  0.060 &  0.152 &  0.230
 \\ 
 & $p=7$ & 0.042 &  0.120 &  0.200  &  0.055 & 0.146 & 0.228 \\ 
 & adaptive & 0.028  & 0.110 & 0.179 &  0.052 & 0.136 & 0.206 \\ 
\vspace{-0.2cm} &  &  &  &  &  &  &  \\ 
$\alpha=0.3$ & $p=3$ & 0.152 &  0.350 & 0.485  & 0.207 & 0.411 
& 0.527 \\ 
 & $p=5$ & 0.178 & 0.369 & 0.488  &  0.242 &  0.440 & 0.542\\ 
 & $p=7$ &  0.178  & 0.366  & 0.480 &  0.223 & 0.411 & 0.534\\ 
 & adaptive & 0.154  & 0.350 & 0.483 & 0.206 & 0.409 & 0.532 \\ 
\vspace{-0.2cm} &  &  &  &  &  &  &  \\ 
$\alpha=0.4$ & $p=3$ & 0.691 & 0.883 & 0.939  &  0.739 & 0.895 & 0.945 \\ 
 &$ p=5$ &  0.704 & 0.886 & 0.942 &  0.749 & 0.908 & 0.954 \\ 
 & $p=7$  & 0.695 & 0.873  & 0.930 &  0.734 & 0.896 & 0.943 \\ 
 & adaptive & 0.708 & 0.889 & 0.944 &  0.726   & 0.900 & 0.944 \\ 
\vspace{-0.2cm} &  &  &  &  &  &  &  \\ 
$\alpha=0.5$ & $p=3$ & 0.994 & 0.999 &  1  & 0.996  &  1  &  1 \\ 
 & $p=5$ & 0.990 &  0.999 &  1  &  0.995 & 0.999 & 1 \\ 
 & $p=7$ &  0.992 & 0.999 & 1  &  0.993 & 0.999   & 1 \\ 
 & adaptive & 0.995  & 1  & 1 &  0.996  &   1 &  1 \\ 
\end{tabular} 
\end{table}

\subsection{Change of covariance operator} We study a time series $(X_n)_{n\in\N}$ in $L^2[0,1]$ following the model
\begin{equation*}
X_i(t)=\begin{cases}\varepsilon_{X,i}(t)\ \ &\text{ for }i<k^\star\\
\varepsilon_{X,i}(t)\left(1+d_1+d_2\big(1+\sin(2\pi t)\big)\right)\ \ &\text{ for }i\geq k^\star,\end{cases}
\end{equation*}
where $(\epsilon_{X,n})_{n\in\Z}$ follows the IID or the FAR(1) model of the previous section. We are interested in testing the hypothesis $d_1=d_2=0$ (stationarity) against the alternative $k^\star\in\{2,\ldots,n\}$ and $(d_1,d_2)\neq 0$. This alternative means that there is a changepoint in the covariance operator. If $d_1\neq 0$ and $d_2=0$, the variance changes uniformly, while for $d_2\neq 0$, we have a nonuniform change of variance. In our simulation, we consider a change in the middle of the data ($k^\star=51$).

As test statistics, we use the supremum type CUSUM statistic $C\!S_n$ and integral type CUSUM statistic $C\! I_n$ given by
\begin{align*}
C\!S_n&:=\sup_{t\in[0,1]}\left\|W_n(t)-tW_n(1)\right\|_{\operatorname{HS}}, \\
C\! I_n&:=\int_{0}^1 \left\|W_n(t)-tW_n(1)\right\|_{\operatorname{HS}}^2dt.
\end{align*}
The two statistics are obviously continuous functionals of the $W_n$ and thus, we can apply Theorems \ref{limitseq} and \ref{bootseq} to obtain bootstrap consistency, meaning that the bootstrap versions
\begin{align*}
C\!S_n^\star&:=\sup_{t\in[o,1]}\left\|W_n^\star(t)-tW_n^\star(1)\right\|_{\operatorname{HS}}, \\
C\! I_n^\star&:=\int_{0}^1 \left\|W_n^\star(t)-tW_n^\star(1)\right\|_{\operatorname{HS}}^2dt.
\end{align*}
will converge weakly to the same limit distribution as the statistics $C\!S_n$ and  $C\! I_n$.

In Table \ref{tab3}, the results for the IID case $\varepsilon_{X,i}=W_{X,i}$ can be found, Table \ref{tab4} contains the results for the FAR(1) case $\varepsilon_{X,i}(t)=\int\Phi(s,t)\varepsilon_{X,i-1}(s)ds+W_{X,i}(t)$. The empirical rejection frequencies are shown for theoretical sizes 1\%, 5\% and 10\% and are based on 3000 simulation runs. For each sample, we create 1000 bootstrap samples to calculate critical values. Both tests hold the size under the hypothesis quite well for the different block lengths. The data-adaptive method by Rice and Shang \cite{rice2017b} leads to a good combination of size and power.

\begin{table}
\caption{Empirical rejection frequencies for $n=100$ and IID model under the hypothesis ($d_1=d_2=0$) and under the alternative of a changepoint at $k^\star=51$ ($d_1\neq 0$ or $d_2\neq 0$) for theoretical sizes 1\%, 5\%, 10\%.}
\begin{tabular}{cc|ccc|ccc|}\label{tab3}
 &  &  & $C\!S_n$ &  &  & $C\!I_n$ &  \\
 & \phantom{MMMM} & 1\% & 5\% & 10\% &1\% & 5\% & 10\% \\ 
\hline   &  & \phantom{MMMM} & \phantom{MMMM} & \phantom{MMMM} \vspace{-0.2cm}& \phantom{MMMM} & \phantom{MMMM} & \phantom{MMMM} \\ 
 $d_1=0,$ & $p=3$ & 0.004 & 0.044 & 0.100 & 0.006 & 0.047 & 0.108 \\ 
$d_2=0$ & $p=5$ & 0.004 & 0.042 & 0.097  & 0.007 & 0.054 & 0.113 \\ 
 & $p=7$ & 0.008 & 0.050 & 0.105  &  0.009 & 0.054 & 0.118 \\ 
 & adaptive & 0.005 & 0.037 & 0.087 &   0.006 & 0.044 & 0.094 \\ 
\vspace{-0.2cm} &  &  &  &  &  &  &  \\ 
$d_1=0.4,$ & $p=3$ & 0.161 & 0.468 & 0.632  &  0.198 & 0.501 & 0.654 \\ 
$d_2=0$ & $p=5$ & 0.114 & 0.406 & 0.588  &  0.149 & 0.453 & 0.609 \\ 
 & $p=7$ & 0.122 & 0.411 & 0.604  &  0.145 & 0.441 & 0.620 \\ 
 & adaptive & 0.151 & 0.448 & 0.620 & 0.183 & 0.474 & 0.629 \\ 
\vspace{-0.2cm} &  & & &  &  &  &  \\ 
$d_1=0.8,$ & $p=3$ & 0.560 & 0.880 & 0.960 &  0.588 & 0.884 & 0.955 \\ 
$d_2=0$ & $p=5$ &  0.404 & 0.819 & 0.933 & 0.459 & 0.835 & 0.931 \\ 
 & $p=7$ & 0.384 & 0.798 & 0.926  &  0.423 & 0.810 & 0.922 \\ 
 & adaptive &0.586 & 0.896 & 0.962 &  0.619 & 0.889 & 0.956  \\ 
\vspace{-0.2cm} &  &  &  &  &  &  &  \\ 
$d_1=0,$ & $p=3$ &0.098 & 0.330 & 0.494 &  0.125 & 0.386 & 0.554 \\ 
$d_2=0.4$ & $p=5$ & 0.072 & 0.284 & 0.459  &  0.100 & 0.346 & 0.534\\ 
 & $p=7$ &   0.084 & 0.301 & 0.474  & 0.107 & 0.356 & 0.532\\ 
 & adaptive &  0.086 & 0.305 & 0.458 &   0.110 & 0.359 & 0.534 \\ 
\vspace{-0.2cm} &  &  &  &  &  &  &  \\ 
$d_1=0,$ &$ p=3$ & 0.458 & 0.829 & 0.938  &  0.503 & 0.858 & 0.946 \\ 
$d_2=0.8$ & $p=5$ &  0.363 & 0.782 & 0.915  &  0.422 & 0.819 & 0.929 \\ 
 & $p=7$  & 0.334 & 0.774 & 0.912  &  0.396 & 0.797 & 0.930\\ 
 & adaptive & 0.501 & 0.852 & 0.946 &  0.540 & 0.880 & 0.954  \\ 
\vspace{-0.2cm} &  &  &  &  &  &  &  \\ 
$d_1=0.4,$ &$ p=3$ & 0.442 & 0.811 & 0.916  & 0.491&  0.822 & 0.921\\ 
$d_2=0.4$ & $p=5$ & 0.321&  0.753 & 0.894  &  0.374 & 0.768 & 0.903 \\ 
 & $p=7$ &  0.303 & 0.727 &  0.880  &  0.353 & 0.746  & 0.880\\ 
 & adaptive & 0.442 & 0.813 & 0.921  & 0.477 & 0.820 & 0.925 \\ 
\end{tabular} 
\end{table}

\begin{table}
\caption{Empirical rejection frequencies for $n=100$ and FAR(1) model under the hypothesis ($d_1=d_2=0$) and under the alternative of a changepoint at $k^\star=51$ ($d_1\neq 0$ or $d_2\neq 0$) for theoretical sizes 1\%, 5\%, 10\%.}
\begin{tabular}{cc|ccc|ccc|}\label{tab4}
 &  &  & $C\!S_n$ &  &  & $C\!I_n$ &  \\
 & \phantom{MMMM} & 1\% & 5\% & 10\% &1\% & 5\% & 10\% \\ 
\hline   &  & \phantom{MMMM} & \phantom{MMMM} & \phantom{MMMM} \vspace{-0.2cm}& \phantom{MMMM} & \phantom{MMMM} & \phantom{MMMM} \\ 
 $d_1=0,$ & $p=3$ & 0.011 & 0.055 & 0.116  & 0.012 & 0.058
 & 0.121 \\ 
$d_2=0$ & $p=5$ &  0.006 & 0.044 & 0.111  &  0.006 & 0.059 & 0.124\\ 
 & $p=7$ &  0.010 & 0.064 & 0.130  &  0.015 & 0.068 & 0.132 \\ 
 & adaptive & 0.006 & 0.044 & 0.097 & 0.008 & 0.052 & 0.110 \\ 
\vspace{-0.2cm} &   &  &   &  &  &  &  \\ 
$d_1=0.4,$ & $p=3$ & 0.196 & 0.496 & 0.660  &  0.231 & 0.529 & 0.679 \\ 
$d_2=0$ & $p=5$ & 0.146 & 0.436 & 0.605  &  0.180 & 0.456 & 0.626 \\ 
 & $p=7$ & 0.148 & 0.437 & 0.631  &  0.180 & 0.480 & 0.641\\ 
& adaptive &0.179 & 0.470 & 0.645 & 0.218 & 0.503 & 0.665  \\ 
\vspace{-0.2cm} &  &  &  &  &  &  &  \\ 
$d_1=0.8,$ & $p=3$ & 0.550 & 0.878 & 0.954  &  0.581 & 0.876 & 0.950 \\ 
$d_2=0$ & $p=5$ & 0.414 & 0.837 & 0.941  & 0.459 & 0.842 & 0.936
 \\ 
 & $p=7$ & 0.416 & 0.814 & 0.934  &  0.456 & 0.820 & 0.929 \\ 
 & adaptive & 0.588 & 0.887 & 0.957  &  0.603 & 0.892 & 0.950 \\ 
\vspace{-0.2cm} &  &  &  &  &  &  &  \\ 
$d_1=0,$ & $p=3$ & 0.109 & 0.330 & 0.492  &  0.139 & 0.378 & 0.547 \\ 
$d_2=0.4$ & $p=5$ & 0.076 & 0.309 & 0.480  & 0.104 & 0.360 & 0.536\\ 
 & $p=7$ & 0.094 & 0.331 & 0.503  &  0.124 & 0.372 & 0.566\\ 
 & adaptive &  0.114 & 0.359 & 0.506 &   0.142 & 0.403 & 0.560  \\ 
\vspace{-0.2cm} &  &  &  &  &  &  &  \\ 
$d_1=0,$ &$ p=3$ & 0.495 & 0.845 & 0.944  & 0.523 & 0.860 & 0.946 \\ 
$d_2=0.8$ & $p=5$ & 0.361 & 0.777 & 0.919  &  0.417 & 0.799 & 0.928 \\ 
 & $p=7$  & 0.350 & 0.776 & 0.916  & 0.413 & 0.798 & 0.921\\ 
 & adaptive &  0.476 & 0.848 & 0.945  &   0.536 & 0.867 & 0.948  \\ 
\vspace{-0.2cm} &  &  &  &  &  &  &  \\ 
$d_1=0.4,$ &$ p=3$ & 0.461 & 0.821 & 0.919  &  0.507 & 0.824 & 0.921\\ 
$d_2=0.4$ & $p=5$ & 0.349 & 0.756 & 0.904  &  0.402 & 0.774 & 0.896 \\ 
 & $p=7$ &  0.324 & 0.746 & 0.889  &  0.386 & 0.765 & 0.885\\ 
 & adaptive & 0.467 & 0.826 & 0.927 &  0.513 & 0.837 & 0.924 \\ 
\end{tabular} 
\end{table}

\pagebreak

\section{Proofs}\label{sec4}

\subsection{Auxilary Results}

\begin{lem}\label{lemMoment} If $X_i$ is a random variable in $H$ and $p\geq 2$ such that $E\left\|X_i\right\|_H^p<\infty$, then
\begin{equation*}
E\left\|(X_i-EX_i)\otimes(X_i-EX_i)\right\|_{\operatorname{HS}}^{p/2}<\infty.
\end{equation*}
\end{lem}

\begin{proof} Obviously, $E\left\|X_i-EX_i\right\|_H^p<\infty$. By the definition of the Hilbert-Schmidt-norm we get for some orthonormal basis $(b_n)_{n\in\N}$ of $H$
\begin{multline*}
\left\|(X_i-EX_i)\otimes(X_i-EX_i)\right\|_{\operatorname{HS}}^2\\
=\sum_{n=1}^{\infty}\langle \langle b_n,(X_i-EX_i) \rangle_H (X_i-EX_i), \langle b_n,(X_i-EX_i) \rangle_H (X_i-EX_i) \rangle_{H}\\
=\langle (X_i-EX_i),(X_i-EX_i) \rangle_H \sum_{n=1}^{\infty} \langle b_n,(X_i-EX_i) \rangle_H , \langle b_n,(X_i-EX_i) \rangle_H \\
=\left\|X_i-EX_i\right\|_H^2\sum_{n=1}^\infty \langle b_n,(X_i-EX_i) \rangle_H^2=\left\|X_i-EX_i\right\|_H^4,
\end{multline*}
where we used the Parseval identity in the last step. So $\left\|(X_i-EX_i)\otimes(X_i-EX_i)\right\|_{\operatorname{HS}}=\left\|X_i-EX_i\right\|_H^2$ and consequently $E\left\|(X_i-EX_i)\otimes(X_i-EX_i)\right\|_{\operatorname{HS}}^{p/2}\leq 2^{p}E\left\|X_i\right\|^{p}$. This implies the statement of the lemma.

\end{proof}

\begin{lem}\label{lemNED} If the sequence $(X_n)_{n\in\Z}$ satisfies  $E\left\|X_n\right\|_H^2<\infty$ and is $L_2$-NED on $(\xi_n)_{n\in\Z}$ with approximation constants $(a_m)_{m\in\N}$, then the sequence $((X_n-EX_n)\otimes(X_n-EX_n))_{n\in\Z}$ is $L_1$-NED on $(\xi_n)_{n\in\Z}$ with approximation constants $(C\sqrt{2a_m})_{m\in\N}$ for some $C<\infty$.
\end{lem}

\begin{proof} Without loss of generality, we can assume that $EX_n=0$. Let $f_m:S^{2m+1}\rightarrow H$ be a function, such that 
\begin{equation*}
E\left\|X_0-f_m(\xi_{-m},\ldots,\xi_m)\right\|^2_H\leq a_m.
\end{equation*}
Define the function $g_m:S^{2m+1}\rightarrow H\otimes H$ by
\begin{equation*}
g_m(\xi_{-m},\ldots,\xi_m)=f_m(\xi_{-m},\ldots,\xi_m)\otimes f_m(\xi_{-m},\ldots,\xi_m),
\end{equation*}
then
\begin{multline*}
E\left[\big\|X_0 \otimes X_0-g_m(\xi_{-m},\ldots,\xi_m)\big\|_{\operatorname{HS}}\right]\\
\leq E\left[\big\|(X_0-f_m(\xi_{-m},\ldots,\xi_m))\otimes X_0\big\|_{\operatorname{HS}}\right]\\
+E\left[\big\|f_m(\xi_{-m},\ldots,\xi_m)\otimes (X_0-f_m(\xi_{-m},\ldots,\xi_m))\big\|_{\operatorname{HS}}\right]
\end{multline*}
Now as in the proof of Lemma \ref{lemMoment}, for any vectors x, y, we have $\|x\otimes y\|_{\operatorname{HS}}=\|x\|_H\|y\|_H$,
so with the help of the H\"older inequality
\begin{multline*}
E\left[\big\|(X_0-f_m(\xi_{-m},\ldots,\xi_m))\otimes X_0\big\|_{\operatorname{HS}}\right]\\
= E\left[\big\|X_0-f_m(\xi_{-m},\ldots,\xi_m)\big\|_H \big\|X_0\big\|_H\right]\\
\leq \sqrt{E\left[\big\|X_0-f_m(\xi_{-m},\ldots,\xi_m)\big\|_H^2\right]}\sqrt{E\left[ \big\|X_0\big\|_H^2\right]}\leq \sqrt{E\left[ \big\|X_0\big\|_H^2\right]} \sqrt{a_m}.
\end{multline*}
Similarly,
\begin{multline*}
E\left[\big\|f_m(\xi_{-m},\ldots,\xi_m)\otimes (X_0-f_m(\xi_{-m},\ldots,\xi_m))\big\|_{\operatorname{HS}}\right]\\
\leq\sqrt{E\left[ \big\| f_m(\xi_{-m},\ldots,\xi_m)\big\|_H^2\right]} \sqrt{a_m}\leq\left(\sqrt{E\left[ \big\|X_0\big\|_H^2\right]} +\sqrt{a_m}\right)\sqrt{a_m}.
\end{multline*}
\end{proof}

\subsection{Proof of the Main Results}

\begin{proof}[Proof of Theorem \ref{limitfixed}] By Lemma \ref{lemMoment} and Lemma \ref{lemNED}, the sequence $((X_n-EX_n)\otimes(X_n-EX_n))_{n\in\Z}$ satisfies the assumptions of Theorem 1.1 by Dehling, Sharipov, Wendler \cite{dehling2015}. So we have
\begin{equation*}
\frac{1}{\sqrt{n}}\sum_{i=1}^n\Big((X_i-E[X_i])\otimes (X_i-E[X_i])-E\big[(X_i-E[X_i])\otimes (X_i-E[X_i])\big]\Big)\Rightarrow N(0,C_\infty),
\end{equation*}
where $C_{\infty}$ is defined in (\ref{LRC}) and $E[(X_i-E[X_i])\otimes (X_i-E[X_i])]=V$. It remains to show that
\begin{equation*}
D_n:=\frac{1}{\sqrt{n}}\sum_{i=1}^n(X_i-E[X_i])\otimes (X_i-E[X_i])-\sqrt{n}\hat{V}_n\rightarrow 0
\end{equation*}
in probability as $n\rightarrow \infty$. By  a short calculation
\begin{multline*}
D_n=\frac{1}{\sqrt{n}}\sum_{i=1}^n(X_i-E[X_i])\otimes (X_i-E[X_i])-\frac{1}{\sqrt{n}}\sum_{i=1}^n(X_i-\bar{X})\otimes (X_i-\bar{X})\\
=\sqrt{n}\left(\big(\bar{X}-E[X_1]\big)\otimes\bar{X}+\bar{X}\otimes\big(\bar{X}-E[X_1]\big)+E[X_1]\otimes E[X_i]-\bar{X}\otimes\bar{X}\right)\\
=\left(\sqrt{n}\big(\bar{X}-E[X_1]\big)\right)\otimes\big(\bar{X}-E[X_1]\big).
\end{multline*}
By Theorem 1.1 in \cite{dehling2015}, $\sqrt{n}(\bar{X}-E[X_1])$ converges in distribution and is thus stochastically bounded, and $(\bar{X}-E[X_1])$ converges to 0 in probability. This completes the proof.
\end{proof}

\begin{proof}[Proof of Theorem \ref{limitseq}]  By Lemma \ref{lemMoment} and Lemma \ref{lemNED}, we can apply Theorem 1 of Sharipov, Tewes, Wendler \cite{sharipov2016} to obtain the weak convergence of $\tilde{W}_n$, defined by
\begin{equation*}
\tilde{W}_n(t):=\frac{1}{\sqrt{n}}\sum_{i=1}^{[nt]}\Big((X_i-E[X_i])\otimes (X_i-E[X_i])-V_X\Big),
\end{equation*}
to $W$. It remains to show that $\sup_{t\in[0,1]}\|\tilde{D}_{[nt],n}\|_{\operatorname{HS}}\rightarrow 0$ in probability as $n\rightarrow\infty$, where
\begin{multline*}
\tilde{D}_{m,n}:=\tilde{W}_n(m/n)-W_n(m/n)\\
=\frac{1}{\sqrt{n}}\sum_{i=1}^m\Big((X_i-E[X_i])\otimes (X_i-E[X_i])-(X_i-\bar{X})\otimes (X_i-\bar{X})\Big).
\end{multline*}
Set $\bar{X}_m=\frac{1}{m}\sum_{i=}^m X_i$. With some calculations, we obtain
\begin{multline*}
\tilde{D}_{m,n}=\frac{m}{\sqrt{n}}\Big((\bar{X}_n-E[X_1])\otimes \bar{X}_m+\bar{X}_m\otimes(\bar{X}_n-E[X_1])+E[X_1]\otimes E[X_1]-\bar{X}_n\otimes\bar{X}_n\Big)\\
=\frac{m}{\sqrt{n}}\Big((\bar{X}_n-EX_1)\otimes (\bar{X}_m-EX_1)+(\bar{X}_m-EX_1)\otimes(\bar{X}_n-EX_1)-(\bar{X}_n-EX_1)\otimes(\bar{X}_n-EX_1)\Big).
\end{multline*}
Now by Theorem 1 in \cite{sharipov2016}
\begin{equation*}
\max_{m=1,\ldots, n}\frac{m}{\sqrt{n}}\big\|\bar{X}_{m}-E[X_1]\big\|_{H}
\end{equation*}
converges weakly to the maximum norm of a $H$-valued Brownian motion and thus this expression is stochastically bounded. As in the proof of Theorem \ref{limitfixed}, $(\bar{X}_n-E[X_1])$ converges to 0 in probability, so
\begin{multline*}
\max_{m=1,\ldots, n}\Big\|\frac{m}{\sqrt{n}}(\bar{X}_n-E[X_1])\otimes (\bar{X}_m-E[X_1])\Big\|_{\operatorname{HS}}\\
=\max_{m=1,\ldots, n}\frac{m}{\sqrt{n}}\big\|\bar{X}_m-E[X_1]\big\|_{H}\big\|\bar{X}_n-E[X_1]\big\|_H\rightarrow 0
\end{multline*}
in probability as $n\rightarrow\infty$. Similar arguments for the other summands of $\tilde{D}_{m,n}$ complete the proof.
\end{proof}

\begin{proof}[Proof of Theorem \ref{bootfixed}] By Lemma \ref{lemMoment} and Lemma \ref{lemNED}, the sequence $((X_n-EX_n)\otimes(X_n-EX_n))_{n\in\Z}$ satisfies the assumptions of Theorem 1.2 in \cite{dehling2015} and we obtain
\begin{equation*}
\frac{1}{\sqrt{pk}}\sum_{i=1}^{pk}\Big((X_i^\star-E[X_i])\otimes (X_i^\star-E[X_i])-E^\star\big[(X_i^\star-E[X_i])\otimes (X_i^\star-E[X_i])\big]\Big)\Rightarrow N(0,C_\infty)
\end{equation*}
almost surely conditional on $X_1,\ldots,X_n$. Without loss of generality, we can assume that $E[X_i]=0$. It remains to show that
\begin{equation*}
D_n^\star:=\frac{1}{\sqrt{pk}}\sum_{i=1}^{pk}\big(X_i^\star\otimes X_i^\star-E^\star\big[X_i^\star\otimes X_i^\star\big]\big)-\sqrt{pk}\big(V_n^\star-\hat{V}_n\big)\xrightarrow{n\rightarrow\infty}0
\end{equation*}
in conditional probability. Let us introduce some notation: Define $\bar{X}_m:=\frac{1}{m}\sum_{i=1}^m X_i$, $\overline{X\otimes X}_m:=\frac{1}{m}\sum_{i=1}^m X_i\otimes X_i$, $\bar{X}_m^\star:=\frac{1}{m}\sum_{i=1}^m X_i^\star$, $\overline{X\otimes X}_m^\star:=\frac{1}{m}\sum_{i=1}^m X_i^\star \otimes X_i^\star$. With this notation, we can write
\begin{equation*}
E^\star\left[X_i^\star\right]=\bar{X}_{kp},\ \ \ E^\star\left[X_i^\star\otimes X_i^\star\right]=\overline{X\otimes X}_{kp}.
\end{equation*}
After some calculations, we arrive at
\begin{multline*}
D_n^\star=\sqrt{kp}\Big(-\overline{X\otimes X}_{kp}+\bar{X}_{kp}^\star\otimes \bar{X}_{kp}^\star+\overline{X\otimes X}_n-\bar{X}_n\otimes\bar{X}_n\Big)\\
=\sqrt{kp}\big(\bar{X}_{kp}^\star-\bar{X}_n\big)\otimes\big(\bar{X}_{kp}^\star-\bar{X}_n\big)+\sqrt{kp}\big(\bar{X}_{kp}^\star-\bar{X}_n\big)\otimes \bar{X}_n\\
+\sqrt{kp}\bar{X}_n\otimes\big(\bar{X}_{kp}^\star-\bar{X}_n\big) +\sqrt{kp}\Big(\overline{X\otimes X}_n-\overline{X\otimes X}_{kp}\Big)=: I_n+I\! I_n+ I\! I\! I_n+ I\! V_n.
\end{multline*}
Note that
\begin{equation*}
\left\|\sqrt{pk}\left(E^\star[X_i^\star]-\bar{X}_n\right)\right\|_{H}\leq \frac{\sqrt{pk}}{nk}\left\| \bar{X}_{kp}\right\|_H+\frac{\sqrt{kp}}{n}\Big\|\sum_{i=kp+1}^n X_i\Big\|_H.
\end{equation*}
By Theorem 1.1 in \cite{dehling2015}, $E\big\|\sum_{i=l+1}^m X_i\big\|_H^2\leq C(m-l)$ for some constant $C$, so with the help of Chebyshev's inequality, we see that $\left\|\sqrt{pk}\left(E^\star[X_i^\star]-\bar{X}_n\right)\right\|_{H}\rightarrow 0$ in probability as $n\rightarrow\infty$. Combining this with Theorem 1.1 in \cite{dehling2015}, we can conclude that
\begin{equation*}
\sqrt{kp}\big(\bar{X}_{kp}^\star-\bar{X}_n\big)=\sqrt{kp}\big(\bar{X}_{kp}^\star-E^\star[X_i^\star]\big)+\sqrt{kp}\big(E^\star[X_i^\star]-\bar{X}_n\big)
\end{equation*}
converges weakly to a normal distribution conditional on $(X_n)_{n\in\N}$. Consequently
\begin{equation*}
I_n=\frac{1}{\sqrt{kp}}\big(\sqrt{kp}\big(\bar{X}_{kp}^\star-\bar{X}_n\big)\big)\otimes\big(\sqrt{kp}\big(\bar{X}_{kp}^\star-\bar{X}_n\big)\big)\rightarrow 0
\end{equation*}
as $n\rightarrow \infty$ in probability conditional  on $X_1,\ldots,X_n$. Because $E\big\|\sum_{i=1}^n X_i\big\|_H^2\leq Cn$, we also have the convergence of $\bar{X}_n$ to 0 in probability by standard arguments, so
\begin{align*}
I\! I_n&=\sqrt{kp}\big(\bar{X}_{kp}^\star-\bar{X}_n\big)\otimes \bar{X}_n\rightarrow  0,\\
I\! I\! I_n&=\bar{X}_n\otimes\big(\sqrt{kp}\big(\bar{X}_{kp}^\star-\bar{X}_n\big)\big)\rightarrow  0
\end{align*}
as $n\rightarrow \infty$ in probability conditional  on $X_1,\ldots,X_n$. Finally, we apply Theorem 1.1 in \cite{dehling2015} to the sequence $(X_n\otimes X_n)_{n\in \Z}$ to obtain $E\|\sum_{i=l+1}^m X_i\otimes X_i\|_{\operatorname{HS}}\leq C_2 (m-l)$ for some constant $C_2<\infty$. Thus
\begin{equation*}
E\left[\big\|I\! V_n \big\|_{\operatorname{HS}}^2\right]\leq 2E\bigg[\Big\|\frac{\sqrt{pk}}{nk} \overline{X\otimes X}_{kp}\Big\|_{\operatorname{HS}}^2\bigg]+2E\bigg[\Big\|\frac{\sqrt{kp}}{n}\sum_{i=kp+1}^n (X_i\otimes X_i)\Big\|_H^2\bigg]\rightarrow 0,
\end{equation*}
and this completes the proof.
\end{proof}

\begin{proof}[Proof of Theorem \ref{bootseq}] By Lemma \ref{lemMoment} and Lemma \ref{lemNED}, we can apply Theorem 2 of Sharipov, Tewes, Wendler \cite{sharipov2016} to obtain the weak convergence of $\check{W}_n$ defined by
\begin{equation*}
\check{W}_n(t):=\frac{1}{\sqrt{pk}}\sum_{i=1}^{[pkt]}\Big((X_i^\star-E[X_i])\otimes (X_i^\star-E[X_i])-E^\star\big[(X_i^\star-E[X_i])\otimes (X_i^\star-E[X_i])\big]\Big)
\end{equation*}
to $W$. Without loss of generality, we can assume that $E[X_i]=0$. As in the proof of Theorem \ref{bootfixed}, we use the notation $\bar{X}_m:=\frac{1}{m}\sum_{i=1}^m X_i$, $\overline{X\otimes X}_m:=\frac{1}{m}\sum_{i=1}^m X_i\otimes X_i$, $\bar{X}_m^\star:=\frac{1}{m}\sum_{i=1}^m X_i^\star$, $\overline{X\otimes X}_m^\star:=\frac{1}{m}\sum_{i=1}^m X_i^\star \otimes X_i^\star$. We have to show that $\sup_{t\in[0,1]}\|\check{D}_{[pkt],h}\|_{\operatorname{HS}}\rightarrow 0$ in conditional probability, where
\begin{multline*}
\check{D}_{m,n}:=\check{W}_n(m/(kp))-W^\star(m/(kp))\\
=\frac{1}{\sqrt{kp}}\sum_{i=1}^{m}\big(X_i^\star \otimes X_i^\star-(X_i^\star-\bar{X}^\star_{kp}) \otimes (X_i^\star-\bar{X}^\star_{kp})+\hat{V}_n-\overline{X\otimes X}_{kp}\big)
\end{multline*}
As in the proof of Theorem \ref{bootseq}, we have $\sqrt{pk}\left(\overline{X\otimes X}_{kp}-\overline{X\otimes X}_{n}\right)\rightarrow 0$ in probability as $n\rightarrow\infty$. Furthermore, we have by the proof of Theorem \ref{limitfixed} that $D_n=\sqrt{n}(\overline{X\otimes X}_{n}-\hat{V}_n)\rightarrow 0$ in probability, so that we can conclude that
\begin{multline*}
\max_{m=1,\ldots, (kp)}\frac{m}{\sqrt{pk}}\left\|\hat{V}_n-\overline{X\otimes X}_{kp}\right\|_{\operatorname{HS}}\\
\leq \sqrt{pk}\left\|\overline{X\otimes X}_{kp}-\overline{X\otimes X}_{n}\right\|_{\operatorname{HS}}+\sqrt{pk}\left\|\overline{X\otimes X}_{n}-\hat{V}_n\right\|_{\operatorname{HS}}\xrightarrow{n\rightarrow\infty}0
\end{multline*}
in probability. Furthermore
\begin{multline*}
\frac{1}{\sqrt{kp}}\sum_{i=1}^{m}\big(X_i^\star \otimes X_i^\star-(X_i^\star-\bar{X}^\star_{kp}) \otimes (X_i^\star-\bar{X}^\star_{kp})\big)\\
=\Big(\frac{1}{\sqrt{kp}}\sum_{i=1}^{m}X_i^\star\Big)\otimes \bar{X}^\star_{kp}+\bar{X}^\star_{kp}\otimes\Big(\frac{1}{\sqrt{kp}}\sum_{i=1}^{m}X_i^\star\Big) +\frac{m}{\sqrt{pk}}\bar{X}^\star_{kp}\otimes \bar{X}^\star_{kp}
\end{multline*}
By Theorem 1.1 and 1.2 of \cite{dehling2015}, we have that
\begin{equation*}
\sqrt{pk}\bar{X}^\star_{kp}=\sqrt{pk}\big(\bar{X}^\star_{kp}-\bar{X}_{kp}\big)+\sqrt{pk}\bar{X}_{kp}
\end{equation*}
is stochastically bounded. Similarly, by Theorem 1 and 2 of \cite{sharipov2016}
\begin{equation*}
\max_{m=1,\ldots,pk}\Big\|\frac{1}{\sqrt{kp}}\sum_{i=1}^{m}X_i^\star\Big\|_{H}\leq \max_{m=1,\ldots,pk}\Big\|\frac{1}{\sqrt{kp}}\sum_{i=1}^{m}\big(X_i^\star-\bar{X}_{kp}^{\star}\big)\Big\|_{H}+\sqrt{pk}\big\|\bar{X}_{kp}^{\star}\big\|_{H}
\end{equation*}
is stochastically bounded. So we can conclude that
\begin{multline*}
\max_{m=1,\ldots,pk}\Big\|\frac{1}{\sqrt{kp}}\sum_{i=1}^{m}\big(X_i^\star \otimes X_i^\star-(X_i^\star-\bar{X}^\star_{kp}) \otimes (X_i^\star-\bar{X}^\star_{kp})\big)\Big\|_{\operatorname{HS}}\\
\leq \frac{1}{\sqrt{pk}}\bigg(2\max_{m=1,\ldots,pk}\Big\|\frac{1}{\sqrt{kp}}\sum_{i=1}^{m}X_i^\star\Big\|_{H}\big\| \sqrt{pk}\bar{X}^\star_{kp}\big\|_{H}+\big\| \sqrt{pk}\bar{X}^\star_{kp}\big\|_{H}^2\bigg)\xrightarrow{n\rightarrow\infty}0
\end{multline*}
in probability conditional on $X_1,\ldots,X_n$. The uniform convergence of $\check{D}_{m,n}$ and thus the statement of the Theorem follows.

\end{proof}

\section*{Acknowledgement} We are grateful to three anonymous referees for their thoughtful comments, which helped to improve this article. We also would like to thank Lea Wegner for her help concerning language corrections. The research was supported by the German Research Foundation (DFG), project WE 5988/3 \emph{Analyse funktionaler Daten ohne Dimensionsreduktion}.

\end{document}